\documentclass[11pt,letterpaper]{amsart}
\usepackage{amssymb, bm}
\usepackage{amsbsy}
\usepackage{amscd}
\usepackage{amsmath}
\usepackage{amsthm}
\usepackage{charter}
\usepackage[active]{srcltx}
\allowdisplaybreaks
\numberwithin{equation}{section}

\def\CC{{\mathbb C}}

\def\PP{{\mathbb P}}
\def\QQ{{\mathbb Q}}
\def\RR{{\mathbb R}}

\def\ZZ{{\mathbb Z}}

\def\intomega{\Omega^{\circ}}

\def\G{{\Gamma}}

\def\spec{{\rm Spec}}

\def\bs{\backslash}
\def\bss{{\bs\!\bs}}

\def\Ecal{{\mathcal E}}

\def\Hcal{{\mathcal H}}

\def\Mcal{{\mathcal M}}

\def\Scal{{\mathcal S}}

\newcommand\Hom{\operatorname{Hom}}

\newcommand\pic{\operatorname{Pic}}

\newcommand\rk{\operatorname{rk}}

\newcommand\GL{\operatorname{GL}}

\newcommand\Res{\operatorname{Res}}
\newcommand\SL{\operatorname{SL}}

\newcommand\Orth{\operatorname{O}}

\newcommand\SU{\operatorname{SU}}

\newtheorem{theorem}{Theorem}[section]
\newtheorem{lemma}[theorem]{Lemma}
\newtheorem{proposition}[theorem]{Proposition}

\newtheorem{corollary}[theorem]{Corollary}

\theoremstyle{definition}
\newtheorem{definition}[theorem]{Definition}

\theoremstyle{remark}
\newtheorem{remark}[theorem]{Remark}

\title[Periods of some hypersurfaces]{Fermat varieties and the periods of some hypersurfaces}

\author[Eduard Looijenga]{Eduard Looijenga}

\address{Mathematisch Instituut\\
Universiteit Utrecht\\ P.O.~Box 80.010, NL-3508 TA Utrecht\\
Nederland}

\email{E.J.N.Looijenga@uu.nl}

\subjclass[2000]{14J70, 14J10}

\keywords{Fermat hypersurface, period map}

\begin{document}

\begin{abstract} 
The variety of all smooth  hypersurfaces of given degree and dimension has 
the  Fermat hypersurface as a natural base point. In order to study the period map for such varieties,  we first determine the integral polarized Hodge structure  of the primitive cohomology of a Fermat hypersurface (as a module over the automorphism group of the hypersurface). We then focus on the degree 3 case and show that the period map for cubic  fourfolds as analyzed by R.~Laza and the author gives complete information about the period map for cubic hypersurfaces of lower  dimension dimension. In particular, we thus recover the results of Allcock-Carlson-Toledo  on  the cubic surface case.
\end{abstract}

\maketitle

\section{Introduction}
For a smooth projective manifold $X$, its polarized Hodge structure is a powerful invariant,
often so strong that its isomorphism type determines the isomorphism type of $X$.
If $X$ is given as a hypersurface in $\PP^n$ of degree $d$, say, then the interesting part of this polarized Hodge structure is the one on the primitive cohomology in degree 
$n-1$,  but it is in general hard to characterize the structures that occur in this manner.  A well-known trick is to look instead at cyclic coverings of $\PP^n$ of degree $d$ branched along $X$: we then get another hypersurface $Y\subset \PP^{n+1}$, still of degree $d$, but now of dimension $n$ and equipped with an action of $\mu_d$. 
The idea is to consider the polarized Hodge structure on the primitive part of $H^n(Y)$ with its $\mu_d$-action instead (or pass to the substructure that remains if we pass to an eigenspace of the
$\mu_d$-action for some primitive character of $\mu_d$). This approach at least 
takes into account that $X$ is not just any complex manifold, but is given as a hypersurface in a projective space. By iterating this procedure we find  a hypersurface in $\PP^{n+k}$
of degree $d$ with an action of $\mu_d{}^k$, or rather, of its semi-direct product with $\Scal_k$.
Reversing this approach, we see that once we have a good understanding of the period map for degree $d$ hypersurfaces in $\PP^N$, then there is a fair chance that we can also attain this for degree $d$ hypersurfaces of lower dimension.  Admittedly this method has its limitations, but we will show here that it works nicely for the first interesting value for $d$, $d=3$, and $N=5$, that is, the case of cubic fourfolds, for which Voisin in her thesis proved a Torelli theorem and for which Laza \cite{laza2} and the author \cite{looijenga:4f}  independently recently characterized the image of the period map. The goal of this paper is to show how this leads to similar characterizations for cubic threefolds and cubic surfaces, thus recovering earlier work of Allcock-Carlson Toledo \cite{act1},\cite{act2} and Looijenga-Swierstra \cite{ls2}. 
The top-down approach that we take here appears in Swierstra's thesis \cite{swierstra}, but  his  presentation is somewhat different.
\\

The paper is organized as follows. The contents of Section 2 might be summed up by saying  that it gives a detailed description of the primitive homology lattice of a Fermat variety together with all the structure that is present:  the intersection form, the Hodge decomposition and the action of the automorphism group of that variety on  this structure. This was certainly known in some way over $\QQ$, but for the purpose of this paper we need it  over $\ZZ$ and we have not been able to find this in the literature. It may thus have an independent interest. A theorem of F.~Pham  on the vanishing lattices of certain singularities proved to be quite useful here.

In Section 3 we determine what happens to the geometric invariant theory (and the markings) of homogeneous forms of degree $d$ if we add to such a form the $d$th power of a new variable. This is useful to know if we want to compare the resulting period maps.

The last two sections deal specifically with cubic hypersurfaces and their period maps:
as explained above, we start out with the result for cubic fourfolds and descend from that
to the cubic threefolds and cubic surfaces. The main results are in \S 5.
\\

I thank Igor Dolgachev for help with Proposition \ref{prop:orbitspacesemb}.  The observation concerning (semi)stability is due to him.
\\

This paper is based on the talk I gave at the conference \emph{Algebraic and Arithmetic Structures of Moduli Spaces}  at Hokkaido University 2007. I thank Professors Nakamura and Weng for organizing this wonderful meeting.

\section{Cohomology of the Fermat varieties}

Throughout this section we fix an integer $d\ge 3$. Since we later take $d=3$,  we suppress $d$ in the notation.
We start out with the Fermat hypersurface in $Y_n\subset \PP^{n+1}$ defined by the equation 
$F_n(Z):=\sum_{i=0}^{n+1} Z_i^d=0$.
If we identify $Y_{n-1}$ with the hyperplane section of $Y_n$ defined by $Z_0=0$, then
$\dot Y_n:=Y_n-Y_{n-1}$ is the affine Fermat variety defined by $\sum_{i=1}^{n+1} z_i^d=-1$. 
Let the group $\mu_d{}^{n+2}$ act diagonally on $\CC^{n+2}$. The projectivized action (on $\PP^{n+1}$)
has as kernel the scalars, that is, the diagonally embedded $\mu_d\subset\mu_d{}^{n+2}$. 
Let us fix a generator $u\in\mu_d$ and denote the corresponding generator of the $i$-th summand in
$\mu_d{}^{n+2}$ by $u_i$, $i=0,\dots , n+1$. So $\mu_d{}^{n+2}/\mu_d$ is the quotient of 
$\mu_d{}^{n+2}$ by the subgroup generated by $u_0u_1\cdots u_{n+1}$. Permutation of coordinates
defines an action of the symmetric group $\Scal_{n+2}$ on $\CC^{n+2}$ and forms with $\mu_d{}^{n+2}$
a semidirect product acting on the same space. The associated action of
$\Scal_{n+2}\ltimes \mu_d{}^{n+2}/\mu_d$ on $\PP^{n+1}$ evidently preserves $Y_n$.
(T.~Shioda, who studied the arithmetical properties of these varieties in depth, determined
among other things the full automorphism group of $Y_n$ in all characteristics, at least  when $(n,d)\not= (2,4)$; 
he finds that in characteristic zero---the case we are dealing with---it is not bigger than  $\Scal_{n+2}\ltimes \mu_d{}^{n+2}/\mu_d$, \cite{shioda}). The $\Scal_{n+2}\ltimes \mu_d{}^{n+2}/\mu_d$-stabilizer of the hyperplane $Z_0=0$ in $\PP^{n+1}$ is $\Scal_{n+1}\ltimes \mu_d{}^{n+1}$ (where $\mu_d{}^{n+1}$ is the group generated by $u_1,\dots ,u_{n+1}$) and so that group acts on the affine Fermat variety $\dot Y_n$. 

\subsection*{The primitive homology of a Fermat variety} 
The primitive homology of $Y_n$, $H^\circ_n(Y)$ is by definition the kernel of $H_n(Y)\to H_n(\PP^{n+1})$.
When $n$ is odd, this is all of $H_n(Y)$. The intersection pairing  $H_n(Y)\otimes H_n(Y)\to \ZZ$ is then antisymmetric and perfect by Poincar\'e duality. When  $n$ is even, this is a sublattice of $H_n(Y)$  of corank one: if $\kappa\in H^2(Y_n)$ denotes the hyperplane class, then $H^\circ_n(Y)$ is the kernel
of the linear form on $H_n(Y)$ defined by $\kappa^{n/2}$. The intersection pairing  $H_n(Y)\otimes H_n(Y)\to \ZZ$ is symmetric and perfect by Poincar\'e duality. The Poincar\'e dual of  $\kappa^{n/2}$ is as a cycle represented by a linear section and has self-intersection $d$. The orthogonal complement of that cycle, $H^\circ_n(Y)$, has therefore a discriminant group cyclic of order $d$. The classical Lefschetz theory shows that this orthogonal complement is generated by vanishing cycles. These have self-intersection $\pm 2$ and so $H^\circ_n(Y)$ is an even lattice.

We think of $\dot Y_n$ as the Milnor fiber of the isolated hypersurface singularity defined by $\sum_{i=1}^{n+1} z_i^d=0$. It is known to have the homotopy type of a bouquet of $n$-spheres. 

We denote by $\kappa\in H^2(Y_n)$ the hyperplane class.
We have the following exact (Gysin) sequence
\[
0\to\! H_{n+1}(Y_n)\to\! H_{n-1}(Y_{n-1})\to\! \tilde H_n(\dot Y_n)\to\! H_n(Y_n)\to\! H_{n-2}(Y_{n-1})\to\! 0,
\]
where $\tilde H$ stands for reduced (co)homology.

Assume first $n=2m$ is even. Then $H_{n+1}(Y_n)=0$ and $H_{n-2}(Y_{n-1})$ is Poincar\'e dual to $\kappa^{m-1}$. This means
that the above sequence yields the exact sequence
\[
0\to H_{n-1}(Y_{n-1})\to \tilde H_n(\dot Y_n)\to H^o_n(Y_n)\to 0,
\]
where $H^o_n(Y_n)$ denotes the primitive homology of $Y_n$, i.e., the kernel of the linear form $\int\kappa^m:H_n(Y_n)\to \ZZ$. Dually, we have the exact sequence
\[
0\to H^n_o(Y_n)\to  \tilde H^n(\dot Y_n)\to H^{n-1}(Y_{n-1})\to 0,
\]
where $H^n_o(Y_n):=H^n(Y)/\ZZ\kappa^m$. Poincar\'e duality identifies $H_n(Y_n)$ with 
$H^n(Y_n)$ and hence $H^o_n(Y_n)$ with $H^n_o(Y_n)$. The composite
\[
j: \tilde H_n(\dot Y_n)\to H^o_n(Y_n)\cong H^n_o(Y_n)\to \tilde H^n(\dot Y_n)
\]
is just the  Poincar\'e duality map of $\tilde H_n(\dot Y_n)$. This shows in particular that 
$H^o_n(Y_n)$ with its intersection pairing can be identified with the 
Milnor lattice $\tilde H_n(\dot Y_n)$ modulo its radical. It is clear that this
holds equivariantly relative to the $\Scal_{n+1}\ltimes\mu_d{}^{n+1}$-action. 

If $n$ is odd, then the above discussion leads to the following exact sequence
$0\to H^\circ_{n-1}(Y_{n-1})\to H_n(\dot Y_n)\to H_n(Y_n)\to 0$
and so  for all $n$ we have the exact sequence
\[
0\to H^\circ_{n-1}(Y_{n-1})\to \tilde H_n(\dot Y_n)\to H^\circ_n(Y_n)\to 0.
\]
We have an intersection pairing on the last two nonzero terms and the 
map between them respects these pairings. Since the one on $H^\circ_n(Y_n)$ is nondegenerate,
we see that the whole sequence can be reproduced from the  intersection pairing
on $\tilde H_n(\dot Y_n)$: $H^\circ_{n-1}(Y_{n-1})$ is its kernel and $H_n(\dot Y_n)$ is its maximal
nondegenerate quotient.

\subsection*{The primitive Fermat lattice} Suppose $L$ is an abelian group, $(a,b)\in L\times L\mapsto a\cdot b\in \ZZ$ a bilinear map and $G$ a \emph{finite abelian} group
acting on $L$ in such a manner as to preserve the pairing: $ga\cdot gb=a\cdot b$. We  regard $L$ as a
$\ZZ G$-module and extend  the pairing sesquilinearly as
\[
(a,b)\in L\times L\mapsto a\star b:=\sum_{g\in G} (a\cdot gb)g\in\ZZ G.
\]
(This means that the pairing is $\ZZ G$-linear in the first variable and antilinear in the second relative to the 
involution `bar' in $\ZZ G$ defined by  $g\in G\mapsto \bar g:=g^{-1}$.) Notice that
the sesquilinear form is \emph{hermitian} ($b\star a=\overline{a\star b}$) if the given form is symmetric and is
\emph{antihermitian} ($b\star a=-\overline{a\star b}$) if it is antisymmetric. 
\\

F.~Pham  determined this sesquilinear extension for $L=H_n(\dot Y_n)$ and $G=\mu_d{}^{n+1}$. 

\begin{proposition}[Pham \cite{pham}]
The group $H_n(\dot Y_n)$ is as a  $\ZZ\mu_d{}^{n+1}$-module generated by a single element $e_n$ that makes it as such isomorphic to the quotient ring 
\[
R_{n+1}:=\ZZ \mu_d{}^{n+1}/\Big(\sum_{k=0}^{d-1}u_\nu^k\Big)_{\nu=1,\dots ,n+1}.
\]
The  sesquilinear extension of the intersection form is $(-1)^{n(n+1)/2}\star$, where $\star$ is given by
\[
e_n\star e_n=(1-\overline{u_1u_2\cdots u_{n+1}})\prod_{\nu=1}^{n+1} (1-u_\nu)\in \ZZ \mu_d{}^{n+1}.\qedhere
\]
Furthermore, $e_n$ transforms under permutation of  
$z_1,\dots ,z_{n+1}$ according to the sign character.
\end{proposition}
\begin{proof}
Pham  actually considers the affine variety $\sum _{\nu=1}^{n+1} z_\nu^d=1$ (rather than $=-1$), but a scalar multiplication by a root of unity of order $2d$ turns this into the statement above.
\end{proof}

 We should perhaps also remark that multiplication by $(1-u)$ identifies  $\ZZ \mu_d/(\sum_{k=0}^{d-1}u^k)$ with
the ideal $(1-u)\ZZ\mu_d$. This shows that the above formula indeed defines a sesquilinear pairing.

In $ \ZZ[\mu_d{}^{n+1}]$ the element $w:=(1-\overline{u_1u_2\cdots u_{n+1}})\prod_{\nu=1}^{n+1} (1-u_\nu)$ satisfies $\overline{w}=(-1)^nw$, which shows that the sesquilinear map in the preceding proposition is hermitian when $n$ is even and antihermitian when $n$ is odd (as it should, since the intersection pairing on $H_n(\dot Y_n)$ is in these cases symmetric resp.\
antisymmetric). Let us also notice that the $\ZZ$-rank of $R_{n+1}$ is $(d-1)^{n+1}$.

\begin{corollary}\label{cor:pham}
The group $H^\circ_n(Y_n)$ is as a  $\ZZ[\mu_d{}^{n+2}/\mu_d]$-module generated by the image $e'_n$ of $e_n$ and thus becomes isomorphic to the quotient ring 
\[
R'_{n+1}:=\ZZ[\mu_d{}^{n+2}/\mu_d]/(\sum_{k=0}^{d-1}u_\nu^k,\; \nu=0,\dots ,n+1)
\]
(which also equals $R_{n+1}/\big(\sum_{k=0}^{d-1}(u_1\cdots u_{n+1})^k\big)$ as a $\ZZ\mu_d{}^{n+1}$-module). The generator $e'_n$ transforms under permutation of the
variables $Z_1,\dots ,Z_{n+1}$ according to the sign character and is up to sign the
only $\ZZ[\mu_d{}^{n+2}/\mu_d]$-module generator with that property. The  sesquilinear extension of the intersection form is $(-1)^{n(n+1)/2}\star$, where $\star$ is given by
\[
e'_n\star e'_n=\prod_{\nu=0}^{n+1} (1-u_\nu)\in \ZZ[\mu_d{}^{n+2}/\mu_d].
\]
Moreover, the kernel of the natural map $R_{n+1}\to R'_{n+1}$ is as a $R_n$-module generated 
$(1-v_n)\sum_{0\le k\le l\le d-2} u_{n+1}^kv_n^le_n$, where $v_n:=u_1\cdots u_n$; this kernel 
is a copy of $R'_n$.
\end{corollary}
\begin{proof}
The natural map  $\mu_d{}^{n+1}\to\mu_d{}^{n+2}/\mu_d$ is an isomorphism of groups (we have $u_0=(u_1\cdots u_{n+1})^{-1}$ in $\mu_d{}^{n+2}/\mu_d$)  and we have a corresponding ring isomorphism between their group rings. Via that identification we
can write $e_n\star e_n=\prod_{\nu=0}^{n+1} (1-u_\nu)\in  \ZZ[\mu_d{}^{n+2}/\mu_d]$.
It is immediate that this pairing factors through $R'_{n+1}$ and it is easy to check that the resulting sesquilinear pairing is nondegenerate. The map $H_n(\dot Y_n)\to H^\circ_n(Y_n)$ has the same property (it forms the
maximal nongenerate quotient) and so via Pham's theorem we   
obtain an identification  $R'_{n+1}e'_n\cong H^\circ_n(Y_n)$.

To prove the last assertion, put $v:=u_1\cdots u_{n+1}=v_n u_{n+1}$ and regard $R'_{n+1}$ as a quotient of
$R_{n+1}$ by the ideal generated by $\sum_{k=0}^{d-1} v^k$. Since $(1-v)(\sum_{k=0}^{d-1} v^k)=0$, this ideal is also the $R_n$-submodule generated by that element. We then find that in $R_{n+1}$
\begin{multline*}
\sum_{k=0}^{d-1} v^k=\sum_{k=0}^{d-1} (v_n u_{n+1})^k
=\sum_{k=0}^{d-2} (v_n u_{n+1})^k -v_n^{d-1}\sum_{k=0}^{d-2} u_{n+1}^k=\\
=\sum_{k=0}^{d-2} u_{n+1}^k(v_n^k-v_n^{d-1})=
(1-v_n)\sum_{0\le k\le l\le d-2} u_{n+1}^kv_n^l.
\end{multline*}
Since $R_{n+1}$ is as a $R_n$-module freely generated by the element $(u_{n+1}^k)_{k=0}^{d-2}$, 
the $R_n$-annihilator of this element is in fact the $R_n$-annihilator of $1-v_n$, that is, $\sum_{k=0}^{d-1} v_n^k$,  and so the $R_n$-submodule generated by it is  isomorphic to $R_n/(1-v_n)R_n=R'_n$.
\end{proof}

\begin{remark}
I do not know whether $H_n^o(Y_n)$ admits a $\ZZ\mu_d{}^{n+1}$-module
generator which transforms under the full permutation group of $Z_0,\dots, Z_{n+1}$ according to the sign character.
\end{remark}

Here is a concrete description for   the bilinear form defined by $\star$. If 
$K:=(k_0,\dots ,k_{n+1})\in (\ZZ/d)^{n+2}$ and 
$I\subset \{0,\dots ,n+1\}$, then abbreviate $u^K:=u_0^{k_0}\dots u_{n+1}^{k_{n+1}}$ and $u_I:=\prod_{i\in I} u_i$. Now $(-1)^{n(n+1)/2} u^K\cdot u^L$ is the coefficient of $1$ in the expression 
$u^{K-L} (1-u_0)\dots (1-u_{n+1})$, where the latter is viewed as an element of the group ring
$\ZZ[\mu_d{}^{n+2}/\mu_d]$. We thus find
\[
(-1)^{n(n+1)/2} u^K\cdot u^L=
\begin{cases}
(-1)^{|I|}\! &\! \text{if\! $u^K=u^Lu_I$ for some nonempty $I$},\\
(-1)^{|I|+n}\! &\!\text{if\! $u_Iu^K=u^L$ for some nonempty $I$},\\
1+(-1)^n &\text{if\! $u^K=u^L$},\\
0 &\text{\! otherwise.}
\end{cases}
\]

\subsection*{A resolution and a rank computation} We have $R'_0=0$ and so repeated application of the previous corollary yields a resolution of $R'_{n+1}$
by $R_1$-modules:
\[
0\to R_1\to R_2\to\cdots \to R_{n+1}\to R'_{n+1}\to 0.
\]
It is clear that  $R_\nu$ is a free $R_1$-module of rank $(d-1)^{\nu-1}$
(we have $R_\nu =R_1\otimes_\ZZ R_{\nu-1}$ as $R_1$-modules and $\rk_\ZZ R_{\nu-1}=(d-1)^{\nu-1}$). So 
\begin{multline*}
\rk_{R_1} H^\circ_n(Y_n)=\rk_{R_1} R'_{n+1}=\\
=\sum_{\nu=1}^{n+1} (-1)^{n+1-\nu}(d-1)^{\nu-1}
=\frac{(d-1)^{n+1}+(-1)^n}{d}.
\end{multline*}
This also yields the familiar formula for $\rk_\ZZ H^\circ_n(Y_n)$ (namely the product of the above number with $(d-1)$).
For example, in the case that will mostly concern us,  $d=3$, we find for $n=1,2,3,4$ for $R_1$-ranks 
the values $1,3,5,11$ respectively and for the $\ZZ$-ranks twice these numbers. 

Let $\zeta_d:=\exp (2\pi \sqrt{-1}/d)$ and consider $\ZZ[\zeta_d]\otimes_{R_1} R'_{n+1}$, where $\ZZ[\zeta_d]$ is
a  $R_1$-module via $u_{n+1}\mapsto \zeta_d$. Then the bilinear form on $R'_{n+1}$ defines
a  sesquilinear form $R'_{n+1}\times R'_{n+1}\to \ZZ[\zeta_d]$  by
\[
(a \cdot b)_1:= \sum_{i\in (\ZZ/d)} (a\cdot u_{n+1}^ib)\zeta_d^i.
\]
The $\ZZ[\zeta_d]$-module $\ZZ[\zeta_d]\otimes_{R_1} R'_{n+1}$ is generated by the element $u^K$, with 
$K=(k_0,\dots ,k_{n})\in (\ZZ/d)^{n+1}$. For these generators our formulae yield
\[
(-1)^{n(n+1)/2} (u^K \cdot u^L)_1=
\begin{cases}
(-1)^{|I|}(1-\bar\zeta_d) \text{  if $u^K=u^Lu_I$ for some $I\not=\emptyset$},\\
(-1)^{|I|+n}(1-\zeta_d) \text{  if $u_Iu^K=u^L$ for some $I\not=\emptyset$},\\
(1-\bar\zeta_d) +(-1)^n(1-\zeta_d) \text{  if $u^K=u^L$},\\
0 \text{  otherwise.}
\end{cases}
\]
The form is hermitian or skew-hermitian according to whether $n$ is even or odd. In the last case we can make it hermitian by multiplication with the purely imaginary number $(\zeta_d+1)(\zeta_d-1)^{-1}$. This suggests to introduce the hermitian forms $h_+$ and $h_-$ by 
\[
h_\pm (u^K,u^L)=
\begin{cases}
(-1)^{|I|}(1\mp\bar\zeta_d) \text{  if $u^K=u^Lu_I$ for some $I\not=\emptyset$},\\
(-1)^{|I|}(1\mp\zeta_d) \text{  if $u_Iu^K=u^L$ for some $I\not=\emptyset$},\\
(2\mp\zeta_d\mp\bar\zeta_d)=(1\mp\zeta_d)(1\mp\bar\zeta_d) \text{  if $u^K=u^L$},\\
0 \text{  otherwise,}
\end{cases}
\]
where we recall that $K,L\in (\ZZ/d)^{n+1}$. If $n$ is even we have $h_+=\pm (\; \cdot\; )_1$ and for
$n$ odd, $h_-=\pm (\zeta_d+1)(\zeta_d-1)^{-1} (\; \cdot\; )_1$. 
(Strictly speaking it is not yet clear that $h_+$ resp.\ $h_-$ is well-defined for $n$ odd resp.\ $n$ even,
but the discussion below will show that this is so.)

An extension of this result will be of interest here. Denote by $\chi_k: R_k\to \ZZ[\zeta_d]$ the character that sends $u_\nu$ to $\zeta_d$,  $\nu=1,\dots ,k$ (we assume $k\ge 1$ here). This makes $\ZZ[\zeta_d]$ an $R_k$-module. For a given integer $m\ge 0$, we consider the exact sequence 
of $R_k$ modules 
\[
R_k\to R_{k+1}\to\cdots \to R_{k+m}\to R'_{k+m+1}\to 0
\]
and tensor it with $\chi_k: R_k\to \ZZ[\zeta_d]$. 
Since $R_{k+l}= R_k\otimes_\ZZ R_l$ is an equality of $R_k$-modules, we get an partial resolution of 
$\ZZ[\zeta_d]\otimes_{R_k}R'_{k+m+1}$ by free $\ZZ[\zeta_d]$-modules:
\[
\ZZ[\zeta_d]\to \ZZ[\zeta_d]\otimes_\ZZ R_1\to\cdots \to \ZZ[\zeta_d]\otimes_\ZZ R_{m}\to 
\ZZ[\zeta_d]\otimes_{R_k}R'_{k+m+1}\to 0.
\]
The first map is induced by multiplication by $(1-v_k)\sum_{0\le j\le l\le d-2} u_{k+1}^jv_k^l$, where $v_k=u_1\cdots u_k$. Therefore it sends $1\in \ZZ[\zeta_d]$ to the element $(1-\zeta_d^k)(\sum_{0\le j\le l\le d-2} \zeta_d^{kl}u_{k+1}^j)$. If $d$ does not divide $k$, then this map is injective,  so that we have in fact a complete resolution. Notice that it
is independent of $k$ (for $k=1$ we have the situation considered above). It follows
that then the $\ZZ[\zeta_d]$-rank of $\ZZ[\zeta_d]\otimes_{R_k}R'_{k+m+1}$ equals 
$d^{-1}\big((d-1)^{m+2}+(-1)^{m+1}\big)$.
We express this somewhat differently as follows. 

\begin{corollary}
Let the group $\Scal_k\ltimes \mu_d{}^k$ act on $\PP^{m+k+1}$ in the obvious manner on the \emph{last} $k$ coordinates (we prefer  this over the first) and denote (also) by $\chi_k$ the character $\Scal_k\ltimes \mu_d{}^k\to \ZZ[\zeta_d]$  that sends every generator of $\mu_d{}^k$ to $\zeta_d$. When $k\notin \ZZ d$, the isomorphism class
of the hermitian $\ZZ[\zeta_d]$-module $\ZZ[\zeta_d]\otimes_{\chi_k}H^\circ_{m+k}(Y_{m+k})$ is independent of $k$  with its associated hermitian form up to sign corresponding to $h_+$ or $h_-$ according to whether $m+k$ is even or odd. In particular, 
\begin{multline*}
\rk_{\ZZ[\zeta_d]}\ZZ[\zeta_d]\otimes_{\chi_k}H^\circ_{m+k}(Y_{m+k})=\\
=d^{-1}\big((d-1)^{m+2}+(-1)^{m+1}\big)=\rk_\ZZ H^\circ_{m}(Y_{m})+ (-1)^{m+1}.
\end{multline*}
\end{corollary}

The  verification of the type of the hermitian form is straightforward; the rest follows from the preceding
discussion.

\subsection*{Characters and  the Griffiths-De Rham description}
We determine the set of characters  of $R'_{n+1}$. If  $\chi: R'_{n+1}\to \CC$ is a ring homomorphism, then 
the fact that $\sum_{k=0}^{d-1} u_i^k=0$ implies that $\chi (u_i)=\zeta_d ^{k_i}$ for some nonzero $k_i\in \ZZ/d$ and so $\chi (u_0u_1\cdots u_{n+1})=\zeta_d^{(\sum_i k_i)}$. The latter must be $1$, and so we will have 
$\sum_i k_i= 0$ in $\ZZ/d$. Conversely,
any tuple $K=(k_0,\dots ,k_{n+1})\in (\ZZ/d-\{ 0\})^{n+2}$ with zero sum gives via the above prescription a character $\chi^K$ of $R'_{n+1}$.
These characters define a basis of  $\Hom_\ZZ(\Ecal'_{n+1},\CC)$ and yield the eigenline decomposition
of that space under the action of the group $\mu_d{}^{n+2}$. This decomposition must be orthogonal with respect to the
sesquilinear form on $R'_{n+1}$.

Since we have an identification $\Hom_\ZZ(R'_{n+1},\CC)\cong H^n_o(Y_n;\CC)$, we can compare the above character basis with one obtained by means of Griffiths' De Rham description \cite{griffiths} of $H_o^n(Y_n;\CC)$: a basis  is here given by certain  
rational differential forms  indexed by the  tuples $L=(\ell_0,\dots, \ell_{n+1})$ with $\ell_i\in\{ 1,\dots ,d-1\}$ and with the property that $d$ divides $\sum_{i=0}^{n+1}\ell_i$: 
\[
\omega _L :=\Res_{Y_n}\Res_{\PP^{n+1}} \Big(\prod_{i=0}^{n+1}Z_i^{\ell_i-1}\Big)\frac{dZ_0\wedge\cdots \wedge dZ_{n+1}}{F_n^{q+1}},
\]
where $q$ is such that the expression is homogeneous of degree zero, that is, given by $q+1=(\sum_i\ell_i )/d$. Moreover, the Griffiths theory tells us that $\omega_L$ lies in $F^{n-q}H^n_o(Y_n;\CC)$, but not in 
$F^{n-q+1}H^n_o(Y_n;\CC)$. For example, if $d=3$ and $n\equiv 1\pmod{3}$, then 
$\omega_{(1^{n+1})}$ spans the bidegree  $(n-[n/3],[n/3])$-part of $H^4_o(Y_4)$.

It is clear that $\omega_L$ is an eigenform for the $\mu_d{}^{n+1}$-action: $u_i$ acts on it as multiplication by $\zeta^{-\ell_i}$ and so the corresponding character is 
$\chi^{-L}$. Since the $\mu_d{}^{n+1}$-action preserves the Hodge structure, it follows that the character basis is compatible with that structure: a small computation shows that 
if we represent $K$ by an integer sequence in $\{1,\dots ,d-1\}$ and denote by $|K|$ its sum,  then $\chi^K$ has Hodge type $(p,n-p)$, where $p=-1+(n+2+|K|)/d$.

\section{Hypersurfaces with partial Fermat symmetry} 

\subsection*{The notion of a $k$-marking} 
The action of $\Scal_{n+1}\ltimes\mu_d{}^{n+1}$ on $\PP^{n+1}$ as above lifts to $\CC^{n+2}$ by letting
the first coordinate unchanged. For $k\le n+1$ we regard $\Scal_k\ltimes\mu_d{}^k$ as the subgroup of
$\Scal_{n+1}\ltimes\mu_d{}^{n+1}$ that acts on the last $k$ coordinates only and we denote this image of
$\Scal_k\ltimes\mu_d{}^k$ in $\GL(n+2)$ by $G_k$. Recall that $\chi_k: G_k\to \CC^\times$ is the character
that sends each $u_i$ to $\zeta_d$. We also use that notation for its restriction to $\mu^k$.

\begin{lemma}\label{lemma: partialFermat}
Any polynomial in $Z_0,\dots ,Z_{n+1}$ homogeneous degree of $d$ and fixed by  $G_k$ 
is a sum of a degree $d$ polynomial  in $Z_0,\dots ,Z_{n+1-k}$ and a scalar multiple of the Fermat form
$F_k:=Z_{n+2-k}^d+\cdots +Z_{n+1}^d$.
\end{lemma}

The easy proof is left to the reader. We will refer to such a  polynomial as a \emph{partial Fermat form of order $k$}. Since such forms are parameterized by a connected variety, we have:

\begin{lemma}
Let $Y\subset\PP^{n+1}$ be a smooth $n$-fold defined by  partial Fermat form of degree $d$ and of order $k$.
Then there is a $G_k$-equivariant isomorphism $H_n(Y)\cong H_n(Y_n)$ which respects the intersection pairing, and (for even $n$)  the class of an $n/2$-dimensional linear section.
\end{lemma}

This suggests the following definition. First, let us agree to call an \emph{abstract hypersurface of degree $d$} a variety with a complete linear system that maps  that variety isomorphically onto a 
hypersurface of degree $d$ (in other words, it comes with an isomorphism onto such a 
hypersurface, given up to projective equivalence).

\begin{definition}\label{def:marking}
Let $Y$ be a smooth abstract hypersurface of degree $d$ and dimension $m$. A \emph{$k$-marking} of $Y$ is given by a pair $(i:Y\hookrightarrow \hat Y,\phi)$, 
where $\hat Y$ is smooth abstract hypersurface of degree $d$ and dimension $m+k$ together with a $G_k$-action of the above type, $i$ is an isomorphism of $Y$ onto the $G_k$-fixed point set in $\hat Y$ (as projective manifolds) and $\phi$ is a $G_k$-equivariant  lattice isomorphism of $H_{m+k+1}(\hat Y)$ onto $H_{m+k+1}(Y_{m+k})$ which, in case $m+k$ is even, takes the $(m+k)/2$-th power of the hyperplane class of $\hat Y$ to that of $Y_{m+k}$. We regard two such markings $(i:Y\hookrightarrow \hat Y,\phi)$ and $(i':Y\hookrightarrow \hat Y',\phi')$ as equal  if there exists a $G_k$-equivariant isomorphism $h:\hat Y\cong\hat Y'$ such that $i'=hi$ and $\phi'g=\phi h^*$ for some $g\in G_k$.
\end{definition}

Notice that a $k$-marking induces a unitary  isomorphism of the following
$\ZZ[\zeta]\otimes _{G_k}H^o_{m+k+1}(\hat Y)$ onto $\ZZ[\zeta]\otimes _{\chi_k}R'_{m+k+1}$.

\begin{remark}
The group $\G$ of transformations of $H_{m+k}(Y_{m+k})$ that respect the intersection product and preserve the relevant class of a linear section is an arithmetic group of orthogonal or symplectic type (depending on the parity of $n$) which contains the image of $G_k$. The $G_k$-centralizer in $\G$,  $\G_k$, is also arithmetic and permutes the $k$-markings of $Y$ transitively. It is easy to see that $G_k$ acts faithfully 
on $R'_{m+k}$. Hence it does so on $H_{m+k}(Y_{m+k})$. So $\mu_d{}^k\subset G_k$ may be regarded a (central) subgroup of  $\G_k$ and the latter acts on the markings via the quotient $\G_k/\mu_d{}^k$.

A period map is defined by assigning to $Y$ the primitive Hodge structure of $\hat Y$ with its $G_k$-action.
A $k$-marking of $Y$ makes this correspond to a $G_k$-invariant Hodge structure on 
$\Hom_{\ZZ}(R'_{n+1},\CC)$. In our examples the most interesting part of that structure  is its  restriction to
the $\bar\chi_k$-character space. That is why we payed special attention to $\ZZ[\zeta_d]\otimes_{\chi_k}R'_{n+1}$.
\end{remark}

\subsection*{Moduli spaces of partial Fermat varieties} 
If $S^d_{n+2}$ is an abbreviation for the $\SL(n+2)$-representation on the space degree $d$ polynomials in the variables $Z_0,\dots ,Z_{n+1}$,  then Lemma \ref{lemma: partialFermat}
asserts that
\[
\big(S^d_{n+2})^{G_k}=S^d_{n+2-k}\oplus\CC.F_k.
\]
Notice that the $G_k$-centralizer of $\SL (n+2)$,
$\SL (n+2)^{G_k}$, can be identified with $\SL (n+2-k)\times\CC^\times$, where the second factor is the space of diagonal matrices of the form $(\lambda I_{n+2-k},\mu I_k)$ with determinant 
$\lambda^{n+2-k}\mu^k =1$. Since this factor affects the coefficient of $F_k$, it follows that we get a map of orbit categorical spaces $\SL (n+2-k)\bss S^d_{n+2-k}\to \SL (n+2)\bss S^d_{n+2}$
whose image is also the image of $\SL (n+2)^{G_k}$ (it takes the origin of $S^d_{n+2-k}$ to the orbit, or rather its closure, of $F_k$). It is the morphism of (affine) quasi-cones defined
by $\CC[S^d_{n+2}]^{\SL(n+2)}\to \CC[S^d_{n+2-k}]^{\SL(n+2-k)}$.

We denote by $St^d_{n+2}\subset sSt^d_{n+2}\subset S^d_{n+2}$ the stable
and the semistable loci. The former contains the locus $S\! m^d_{n+2}$ which parametrizes the smooth $n$-folds.

\begin{proposition}\label{prop:orbitspacesemb}
The morphism of quasi-cones 
\[
\SL (n+2-k)\bss S^d_{n+2-k}\to \SL (n+2)\bss S^d_{n+2}
\] 
is finite  and is a normalization of its image (equivalently,
the image of the embedding $\CC[S^d_{n+2}]^{\SL (n+2)}\to \CC[S^d_{n+2-k}]^{\SL (n+2-k)}$ is an integral extension of $\CC[S^d_{n+2-k}]^{\SL (n+2-k)}$ in the latter's field of fractions). 
It sends stable (resp.\ semistable) orbits to stable (resp.\ semistable) orbits.
\end{proposition}
 \begin{proof}
Observe that the morphism in question is generically injective and that its
 domain and range are normal. We prove that when $k=1$, it sends (semi)stable points to (semi)stable points; the proposition then follows with induction on $k$.
 
For this we invoke the well-known simplex criterion. Assign a monomial of degree $d$ in $m$ variables the vector in $\ZZ^m=\ZZ e_1+\cdots +\ZZ e_m$ whose components are the exponents of the monomial. This makes the monomials in $S^d_m$ correspond to the integral vectors in the simplex $\Delta_m$ spanned by $de_1,\dots ,de_m$. Then an element of $F\in S^d_m$ is semistable resp.\ stable if and only if the convex hull of the integral vectors in $\Delta_m$ representing monomials appearing in $F$ contains the barycenter $\frac{d}{m}\sum_i e_i$ of
 $\Delta_m$ resp. contains the barycenter in its interior. Since the convex set attached to 
 $F+X_{m+1}^d$ is the cone (with vertex $e_{m+1}$) over the convex set attached to $F$, the assertion follows easily. 
\end{proof}

This allows us to regard a projective equivalence class of stable $(n-k)$-folds of degree $d$ as a projective equivalence class of stable $n$-folds of degree $d$ with a $G_k$-action
of the type described above. More precisely, it associates to a  $(n-k)$-fold of degree $d$ an  $n$-fold of degree $d$ that is given up to an action by $G_k.\CC^\times$.

\begin{remark}
By definition a $k$-marked $m$-fold of degree $d$ defines a marked $(m+k)$-fold of that same degree
up to $G_k$-symmetry.
This notion is useful if we happen to have a good understanding of the period map for the latter.
In that case it is desirable to characterize via the period map the locus of $k$-marked $m$-folds in the moduli space of marked $(m+k)$-folds, for instance as the locus where the Hodge structure has a $G_k$-symmetry.
This is particularly so when the period map is a local isomorphism. If we assume $n$ even, then this only  happens for  cubic fourfolds and quartic surfaces but it at least helps us to get hold of the period maps for cubic threefolds, cubic surfaces and quartic curves. The period map for quartic surfaces is essentially that of K3 surfaces of degree four and the derived period map for quartic curves was treated by Kond\=o \cite{kondo} (see also Looijenga \cite{looijenga:quartics} and Artebani \cite{art}). 
Given its share of attention, we shall here focus on the cubic  case.
\end{remark}

\section{A period map for a cubic hypersurface of dimension $\le 4$}

We now take $n=4$, $d=3$. So $G_k$ stands for the subgroup $\Scal_{k}\ltimes\mu_3{}^{k}\subset \GL (6)$ acting on the last $k$ coordinates. We write $\Ecal$ for $R_1=\ZZ[u]/(u^2+u+1)$ and sometimes
refer to its elements as \emph{Eisenstein integers} (it is the ring of integers in the corresponding quadratic extension of $\QQ$). We  also write $\Lambda$ for the lattice $H_4(Y_4)$ with its action of
$G_5$.

The lattice $\Lambda$ is unimodular, odd and of signature $(21,2)$. 
The square of the hyperplane class, denoted $\eta\in \Lambda$, has the property that $\eta\cdot\eta=3$ and that its orthogonal complement, denoted $\Lambda_o\subset \Lambda$, is even (the self-product of every vector is even). Moreover, $\ZZ\eta$ is the fixed point sublattice of the  action of $G_5$.
We put $V:=\Hom(\Lambda_o,\CC)$ and notice that $V$ has an underlying $\QQ$-structure with inner product.
The group of lattice automorphisms of $\Lambda$ that fix $\eta$ is denoted by $\G$.

\subsection*{The period domains} 
If $Y\subset \PP^5$ is any smooth cubic hypersurface, then the Hodge numbers in the middle dimension are
$h^{3,1}(Y)=h^{1,3}(Y)=1$ and $h^{2,2}(Y)=21$. The Griffiths theory shows that a generator of $H^{3,1}(Y)$ is determined by an equation $F\in\CC[Z_0,\dots. Z_5]$ for $Y$, namely
\[
\Res_Y\Res_{\PP^5} \frac{dZ_0\wedge\cdots \wedge dZ_5}{F(Z_0,\dots ,Z_5)^2}.
\]
We denote this class $\omega (F)\in H^4(Y;\CC)$. It satisfies the familiar equality $\omega(F)\cdot \omega(F) =0$ and the inequality $\omega(F)\cdot\overline{\omega}(F) <0$. 

If $Y=Y_4$ and take for $F$ its Fermat equation $F_6$, then
we see that $\omega (F_6)$ transforms under $G_5$ as its numerator 
$dZ_0\wedge\cdots \wedge dZ_5$: it is invariant under $\Scal_5$ and $u_i$ acts as multiplication by $\bar\zeta_3$.
We denote this character $\chi$ and its restriction to $G_k$ by $\chi_k$. We further denote by $V_k$ the 
$\chi_k$-eigenspace in $V$.

\begin{lemma}\label{lemma:signatures}
The $\chi_k$-eigenspace $V_k$ is for $k=1,2,3$
a subspace on which the Hermitian form defined by 
$h(\omega_1,\omega_2):=\omega_1\cdot\bar\omega_2$ is nondegenerate and of hyperbolic signature
$(10,1)$, $(4,1)$, $(1,1)$ respectively. 
\end{lemma}
\begin{proof}
It is clear that any linear map $f:H^o_4(Y_4)\cong \Ecal'_5\to \CC$ which transforms according to the character $\chi$ is $f(1)$ times the ring homomorphism defined by $u_i\mapsto\zeta_3$. 
So  the $\chi$-eigenspace of $G_5$ in $V=H_o^4(Y_4;\CC)$ is one-dimensional (and hence spanned by $\omega (F_6)$). The   $\overline{\chi}$-eigenspace is then spanned by $\bar\omega (F_6)$. We know from Hodge theory that  $\omega (F_6)$ and $\bar \omega (F_6)$ span a maximal negative definite subspace of $V$ for $h$. This implies that for $k\le 3$, $V_k$ is is nondegenerate and of hyperbolic signature. It remains
to compute the dimension of $V_k$. This is straightforward.
\end{proof}

We introduce the period domain
\[
\Omega:=\{ \omega\in V\, :\, \omega\cdot \omega =0, \omega\cdot\overline{\omega}<0\}
\]
and put $\Omega_k:=\Omega\cap V_k$. The following lemma is a formal consequence of  Lemma \ref{lemma:signatures}
and we therefore omit its proof. 

\begin{lemma}
The projectivization $\PP\Omega\subset \PP(V)$ has two connected components that are interchanged by complex conjugation; each of these is a 
$20$-dimensional symmetric domain for the orthogonal group $\Orth(\Lambda_o\otimes\RR)$. The group $\G$ is
arithmetic in $\Orth (\Lambda_o\otimes\RR)$ so that $\G\bs\Omega$ has the  structure of a Shimura variety.

For  $k=1$, $2$, $3$, $\PP\Omega_k$ is a complex ball of dimension $10$, $4$, $1$ respectively and a  symmetric domain for the special hyperbolic unitary group $\SU(V_k)$.
The centralizer of $G_k$ in $\G$ (which we shall denote by $\G_k$) is arithmetic in the centralizer of $\Scal_{k}\ltimes\mu_3{}^{k}$ in $\Orth (\Lambda_o\otimes\RR)$ and 
acts as such (hence properly discontinuously) in  $\PP\Omega_k$ so that
$\G_k\bs  \PP\Omega_k$ has the  structure of a Shimura variety.
\end{lemma}

\subsection*{The definition of the period map} 
We first characterize a partial Fermat cubic's symmetry  in cohomological terms.

\begin{proposition}\label{prop:marking}
Let $Y$ be a cubic fourfold endowed with an action of $\Scal_k\ltimes\mu_3{}^k$ such that the fixed point set of 
$u_1\in\mu_3{}^k\subset \Scal_k\ltimes\mu_3{}^k$ in  $H_o^4(Z)$ is trivial. Then $Y$ with its $\Scal_k\ltimes\mu_3{}^k$-action is projectively equivalent to a partial Fermat cubic of order $k$ (with $\Scal_k\ltimes\mu_3{}^k$ acting as $G_k$).
\end{proposition}
\begin{proof}
Since $\pic(Y)$ is generated by the hyperplane class $\kappa_Y$,  $\Scal_k\ltimes\mu_3{}^k$ acts projectively.
We first show that $u_1$ fixes a hyperplane $H_1\subset\PP^5$ pointwise. To this end we compute its Lefschetz number.
Our assumption implies that the characteristic polynomial of $u_1$ in $H_o^4(Y;\CC)$ is $11(T^2+T+1)$ 
so that its trace is $-11$. Since the powers $1,\kappa_Z,\dots ,\kappa_Z^4$ of the hyperplane class are fixed by $u_i$, it follows that the Lefschetz number of $u_i$  equals $-11+5=-6<0$.  So by  Lefschetz theorem, the Euler characteristic of the fixed point set of $u_i$ is negative. An eigenspace of $u_1$ in $\PP^5$ of
dimension $d$ meets $Y$ in linear section $Y_1$ of dimension $\ge d-1$.  
If $d\le 3$, $Y_1$ is a finite set, a line, a cubic curve or a cubic surface. Since it is also smooth, it follows that $Y_1$ has positive Euler characteristic. Hence $u_1$ has an eigen space $H_1$ of dimension $4$. 
Since $u_i$ is conjugate to $u_1$ in $\Scal_k\ltimes\mu_3{}^k$, it fixes a  hyperplane $H_i\subset\PP^5$ pointwise.

The $u_i$ pairwise commute and act differently in $H^4(Y)$, and so the hyperplanes $H_i$ are linearly independent. They are permuted by $\Scal_k$ via its tautological action on the index set.
It then easily follows that we can find coordinates $[W_0:\cdots :W_5]$ for $\PP^5$ such that $\Scal_k$ acts as permutation group of the last $k$ coordinates (here $W_j$ has the index $6-j$)
and $u_1$ acts as $u_1[W_0:\cdots :W_5]=[W_0:\cdots :W_4:\zeta_3W_5]$. 
\end{proof}

It follows from Proposition \ref{prop:marking} that the $h$  in the definition of a $k$-marking of a cubic 
$(4-k)$-fold will be unique up to
a composition with an element of $G_k$.  We also observe that these markings are permuted simply transitively by $\G_k/G_k$. The pairs $(F,\phi)$ with $F\in Sm^3_{6-k}$ and $\phi$ a marking of the associated smooth cubic  
define an unramified $\G_k/G_k$-covering $\widetilde{Sm}{}^3_{6-k}\to Sm^3_{6-k}$.

The projective automorphism group of a smooth cubic  $4$-fold acts faithfully on its set of markings and 
so the marked smooth cubic  $(4-k)$-folds in $\PP^{5-k}$. The scalars $\CC^\times\subset\GL(6-k)$ act  on the
first component via their third power. Hence the action of $\GL(6-k)$ on $Sm^3_{6-k}$ lift to one on
$\widetilde{Sm}{}^3_{6-k}$ and the latter factors through a free action of $\GL(6-k)/\mu_3$.

If $Y$ is given in $\PP^5$ by an equation $F$, then a marking $\phi$ of $Y$ defines an element $\Omega$, so that we have a map
\[
\widetilde{\Psi}: \SL(6-k)\bs \widetilde{Sm}{}^3_{6-k}\to G_k\bs \Omega_k.
\]
Here is $\Omega_0=\Omega$ and $G_0=1$, whereas for $k>0$,  $G_k$ acts on $ \Omega_k$ via $\chi_k$.
Without the marking we only get an element of $\G\bs \Omega$: we have a map
\[
\Psi_k : \SL(6-k)\bs Sm^3_{6-k}\to \G_k\bs \Omega_k.
\]
Domain and range of $\Psi_k$ come with a $\CC^\times$-action: we let $\lambda\in \CC^\times$ act on 
$P_{6-k}$  by scalar multiplication and on $V$ by inverse scalar multiplication (which by duality corresponds to
ordinary scalar multiplication on the \emph{homology} $H_4(Y_o;\CC)$). These determine proper
$\CC^\times$-actions in the orbit space $\SL(6-k)\bs \widetilde{Sm}{}^3_{6-k}$ and in $\G_k\bs \Omega_k$.
These actions are not always faithful, so let us refer to the action of the image groups 
(which are copies of $\CC^\times$) as the \emph{reduced $\CC^\times$-action}.

\begin{proposition}\label{prop:equiv}
The period map $\Psi_k$ is equivariant for the reduced $\CC^\times$-actions on domain and range.
\end{proposition}
\begin{proof}
Let us do the case $k=0$ first.
If $\lambda\in\CC^\times$ and $G\in Sm^3_6$, then $\Psi (\lambda G)=\lambda^2\Psi (G)$, where we bear in mind
that $\CC^\times$ acts on $\Omega$ by inverse scalar multiplication. So $\Psi$ has degree $2$.
But $\SL(6)$ meets the scalars in the $6$th roots  of unity and these act in $S^3_6$ through $\pm 1$. On the other hand $\G$ does not contain $-1$, for that would mean that minus the orthogonal reflection in $\eta$ preserves $\Lambda$ and this is clearly not the case. So $\Psi$ is equivariant for the reduced $\CC^\times$-actions. 

Let us now do the cases $k=1,2,3$. If $F\in Sm^3_{6-k}$ and $F_k=Z_{7-k}^3+\cdots +Z_6^3$, then
$(\lambda^{-1} I_{6-k},I_k)\in \GL (6)$ takes the cubic defined by $F+F_k$ to the cubic defined by
$\lambda^3F+F_k$ and this isomorphism takes $\omega (F+F_k)$ to  
$\lambda^{6-k}\omega (\lambda^3F+F_k)$. So $\Psi_k(\lambda^3F)=\lambda^{6-k}\Psi_k(F)$. This shows
that $\Psi_k$ is in an algebraic sense homogeneous of degree $(6-k)/3$. However, the center of $\SL (6-k)$,
which consists of the $(6-k)$th roots of unity, acts on $S^3_{6-k}$ through 
$\zeta\in\mu_3{}^{6-k}\mapsto \zeta^{-3}\in\mu_3{}^{6-k}$ (for $S^3_{6-k}$ is the third symmetric power of the dual of the tautological representation of $\SL (6-k)$). This means that when $3$ is relatively prime to $6-k$, i.e., when $k=1,2$, we have a factorization through $\mu_3{}^{6-k}\bs S^3_{6-k}$. On the other hand, the center of $\G_k$ is $\mu_3$. It follows that in all these cases, $\Psi_k$ is equivariant for the reduced $\CC^\times$-actions.  
\end{proof}

The more customary  form of the period  map for cubic $4$-folds is the projectivization of $\Psi$,
\[
\PP\Psi : \SL (6)\bs \PP Sm^3_6\to \G\bs \PP\Omega.
\]
A deep theorem of Claire Voisin  asserts that this map is an open embedding (another proof is given in \cite{looijenga:4f}). The openness, which amounts to a local Torelli theorem, is easy and was known before; the hard part is the injectivity. The last property is more directly stated as follows:

\begin{theorem}[Voisin \cite{voisin}] 
Let $Y$ and $Y'$ be smooth cubic $4$-folds. Then any lattice isomorphism $H^4(Y')\cong H^4(Y)$ that sends hyperplane class to hyperplane class and respects the Hodge decomposition (briefly, a \emph{Hodge isometry}),
is induced by a unique projective  isomorphism $Y\cong Y'$.
\end{theorem}

If we combine this with Proposition \ref{prop:equiv}, then it follows that $\Psi$ itself is also an open embedding. According to \ref{prop:orbitspacesemb}, $\Psi_k$ is a restriction of $\Psi$ and so the same is true for $\Psi_k$.
Our reason for considering this slight refinement of the classical period map is that it maps automorphic forms
on $\PP\Omega_k$ to $\CC[P_{6-k}]^{\SL(n-k)}$. For a {$\G_k$-automorphic form on 
$\Omega_k$ of weight $d$} is simply a $\G_k$-invariant analytic function $f:\Omega_k\to \CC$ that is homogenenous of degree $d$ relative to the $\CC^\times$-action inverse to scalar multiplication
(so $-d$ relative to scalar multiplication) that is subject to certain growth conditions (which are in fact empty unless
$k=3$). It is clear that $\Psi_k$ takes such a function to one on $Sm^3_{6-k}$ which is complex-analytic, homogeneous and
$\SL (6-k)$-invariant. We will see that the latter is always the restriction of a $\SL (6-k)$-invariant homogeneous  polynomial on $S^3_{6-k}$.

\section{The fundamental arithmetic arrangement}

We recall from \cite{looijenga:4f} the following notion.

\begin{definition}
A vector $v\in\Lambda_o$ is called \emph{special} if $v\cdot v=6$  and $v-\eta$ is divisible by $3$ and is called
\emph{nodal} $v\cdot v=2$. A hyperplane in $V$ is called \emph{special} resp.\ \emph{nodal} if it is  perpendicular to such a vector.
\end{definition}
 
A simple argument in lattice theory shows that the $3$-divisibility of $v-\eta$ can be expressed in terms of $\Lambda_o$ only:
it is equivalent to requiring that the inner product of $v$ with any element of $\Lambda_o$ 
lies in $3\ZZ$. Notice that if $-v+\eta=3e$, then $e\cdot e=e\cdot\eta=1$.

We denote the collection of special resp.\ nodal hyperplanes by $\Hcal_s$ resp.\ $\Hcal_n$. It is shown in \cite{looijenga:4f} (the proof is not difficult) 
that each of these  is a single $\G$-orbit. It is clear that every such hyperplane has signature $(19,2)$ and hence  will meet $\Omega$. The union of these two collections of hyperplanes is locally finite in $\Omega$. So the complement of the union of these hyperplanes in $\Omega$ is an open subset $\intomega\subset\Omega$. 
The following theorem 
is independently due to  Laza \cite{laza1}, \cite{laza2} and the author \cite{looijenga:4f}. 

\begin{theorem}\label{thm:surj}
The map $\tilde{\Psi}$ maps $\SL(6)\bs \widetilde{Sm}{}^3_{6}$ isomorphically onto the arrangement complement $\intomega$, and hence  $\Psi$ maps $\SL(6)\bs Sm^3_{6}$ isomorphically onto $\G\bs\intomega$.
\end{theorem}

This immediately leads to corresponding results for cubics of lower dimension.
The case of cubic surfaces is due to Allcock-Carlson-Toledo \cite{act1}, and the case of cubic threefolds was obtained independently by Allcock-Carlson-Toledo \cite{act2} and Looijenga-Swierstra \cite{ls1} (the case of cubic curves is of course classical). The uniform approach that we take here was developed essentially (although somewhat differently) in the thesis of Swierstra \cite{swierstra} and provides among other things a new proof for the case of cubic surfaces.
 
Let us write $\intomega_k$ for  $\intomega\cap V_k$.

\begin{corollary}
For $k\ge 1$, the period map $\tilde\Psi_k$ defines an isomorphism of $\SL(6-k)\bs\widetilde{Sm}{}^3_{6-k}$ onto $\mu_3\bs\intomega_k$ and (hence) drops to an isomorphism  $\Psi_k$ of $\SL(6-k)\bs Sm^3_{6-k}$ onto
$\G_k\bs\intomega_k$.
\end{corollary}
\begin{proof}
It is clear that $\tilde{\Psi}_k$ takes values in $\intomega_k$. 

If the marked
the cubic forms $F,F'\in Sm^3_{6-k}$  have the same image in $\intomega_k$, then by Voisin's Torelli theorem there is a $g\in \SL(6)$ that takes $F+F_k$ to $F'+F_k$ such that the induced isomorphism
$Y(F+F_k)\cong Y(F'+F_k)$ is compatible with the markings. The latter is in particular equivariant for
the actions of $G_k$ on the cohomology of these $4$-folds. Since the projective automorphism group of
a smooth cubic fourfold is faithfully represented on its cohomology, it follows that
the isomorphism in question is already $G_k$-equivariant. This means that the marked cubic forms  lie in the same $\SL (6-k)$-orbit. 

As for surjectivity, if $\omega\in\intomega_k$, then by Theorem \ref{thm:surj}, $\omega$ is represented
by a marked cubic form $(F,\phi)$. Since $\omega$ transforms via $G_k$ under the character $\chi_k$, it
follows that $F$ is $G_k$-invariant. Hence  the last $k$ variables of $F$  separate as  multiple of $F_k$. So $(F,\phi)$ naturally underlies a marked cubic $(4-k)$-fold.
\end{proof}

Denote by $S\! f^3_6\subset S^3_6$ the space of cubic forms with isolated singularities with finite local monodromy groups. According to Arnol'd the singularities with that property are analytically isomorphic to singularities with a local equation of the form $f(z_1,z_2,z_3)+z_4^2+z_5^2$, where $f$ defines a Kleinian singularity. They are indexed by  the Weyl group labels $A_*, D_*, E_*$, as these describe  the local monodromy groups.  It is known that the period map extends to  $S\! f^3_6\subset S^3_6$. 
According to Yokoyama \cite{yokoyama} and Laza \cite{laza1}  such a cubic form is stable:
$S\! f^3_6\subset St^3_6$. 

In the cited references Laza and the author also show: 

\begin{theorem}\label{thm:main}
Denote by $\Omega^f$ the complement in $\Omega$ of the union of the hyperplanes in 
$\Hcal_s$. Then the period map $\Psi$ extends to an isomorphism from  
$\SL(6)\bs S\! f^3_{6}$ onto $\G\bs\Omega^f$.
\end{theorem}

We use the theory developed in \cite{looijenga:arr} to refine and generalize the above theorem.

Write $\Omega^f_k$ for  $\Omega^f\cap\Omega_k$. It is a priori clear that 
the period map induces a map $\Psi_k:\SL(6-k)\bs S\! f^3_{6-k}\to \G_k\bs\Omega^f_k$.
This implies that $\Omega_k$ is not contained in a special hyperplane. 
Denote by $A_d(\Omega^f_k)^{\G_k}$ the space of $\G_k$-automorphic forms of weight $d$ on  $\PP\Omega^f_k$, that is, the space of $\G_k$-invariant complex-meromorphic functions on $\PP\Omega_k$  of weight $d$ that are regular on $\PP\Omega^f_k$, subject to a growth condition when $k=3$ (in which case $\PP\Omega^f_3=\PP\Omega_3$ is the unit disk). Here $d$  may be any integer.  The direct sum of these spaces forms a graded $\CC$-algebra $A_\bullet(\Omega^f_k)^{\G_k}$. In what follows we exclude the somewhat special classical case $k=3$ (for which the period map is one between one dimensional varieties). 

It is an easy fact, and proved in \cite{looijenga:4f},  that if a positive semidefinite sublattice of $\Lambda_o$ is spanned by special vectors, then its rank is $\le 2$. Together with the 
inequality $\dim\Omega_k > 5$ this implies, according to \cite{looijenga:arr},  the following lemma (which for $k=0$ is in fact a step in our proof of Theorem \ref{thm:main}).

\begin{lemma}
For $k=0,1,2$ the $\CC$-algebra $A_d(\Omega^f_k)^{\G_k}$ is normal with finitely many generators in  positive degree  and the analytic map
\[
\G_k\bs\Omega^f_k\to \spec(A_d(\Omega^f_k)^{\G_k})
\]
is an open embedding with complement a subcone of codimension $> 1$. 
\end{lemma}

So $\G_k\bs\Omega^f_k$ has the structure of a quasi-affine variety with $\CC^\times$-action and $A_d(\Omega^f_k)^{\G_k}$ can be understood as its (graded) algebra of regular functions.

\begin{corollary}\label{cor:sfperiod}
For $k=0,1,2$ the period map  defines an isomorphism 
$\Psi_k: S\! f^3_{6-k}\to \G_k\bs\Omega^f_k$ and identifies the $\CC$-algebra 
$A_\bullet(\Omega^f_k)^{\G_k}$ with the 
$\CC$-algebra $\CC[S^3_{6-k}]^{\SL (6-k)}$, but rescales  the degrees by a factor $(6-k)/3$.
\end{corollary}
\begin{proof} The map $\SL(6-k)\bs S\! f^3_{6-k}\to \G_k\bs\Omega^f_k$  is a birational, finite morphism between normal varieties. So it is an isomorphism. 
We have that  the complement of  $\SL(6-k)\bs S\! f^3_{6-k}$ in $\SL(6-k)\bss S^3_{6-k}=\spec (\CC[S_{6-k}]^{\SL (6-k)})$ is of codimension $>1$. Similarly, the complement of 
$\G_k\bs\Omega^f_k$ in  $\spec(A_\bullet(\Omega^f_k)^{\G_k})$ is  of codimension $>1$.
Hence $A_\bullet(\Omega^f_k)^{\G_k}$ and $\CC[S_{6-k}]^{\SL (6-k)}$ can be characterized as the algebra of regular functions on these varieties and so the second statement also follows. The rescaling  of the grading  by the factor $(6-k)/3$ is explained 
by the following equality:
if $F\in S^3_{6-k}$, then if we put $\tilde z_i:=t^{1/3}z_i$ we have
\begin{multline*}
\omega (tF+F_k)=\frac{dz_0\wedge\cdots \wedge dz_5}{(tF(z_0,\dots ,z_{5-k})+
F_k(z_{6-k},\dots z_5))^2}=\\
t^{-(6-k)/3}\frac{dz_0\wedge\cdots\wedge dz_{5-k}\wedge d\tilde z_{6-k}\wedge\cdots \wedge  d\tilde z_5}{(F(z_0,\dots ,z_{5-k})+F_k(\tilde z_{6-k},\dots, \tilde z_5))^2}
=t^{-(6-k)/3}\omega(F+F_k).
\end{multline*}
\end{proof}

The next proposition implies that in the case of cubic surfaces, the GIT compactification coincides with the Baily-Borel compactifications, a theorem that is due
to Allcock-Carlson-Toledo \cite{act1}:

\begin{proposition}\label{prop:csperiod}
The domain  $\Omega_2$ does not meet a special hyperplane so that
we have an isomorphism  $\Psi_2: S\! f^3_4\to \G_k\bs\Omega_2$ and identifies the $\CC$-algebra $A_\bullet(\Omega_2)^{\G_2}$ with the 
$\CC$-algebra $\CC[S^3_4]^{\SL (4)}$.
\end{proposition}
\begin{proof}[Sketch of proof]
There is only one strictly semistable closed orbit of cubic surfaces, namely the one with
three $A_2$ singularities, represented by $Z_0^3-Z_1Z_2Z_3=0$. It corresponds to a 
cubic fourfold with three singularities of type $\tilde E_6$. The period map goes near these singularities to the Baily-Borel boundary. This means that we actually get a proper map
from $\SL(4)\bs S\! f^3_{4}$ to $G_k\bs\Omega_2$.  This can only happen if $\Omega_2=\Omega^f$. In other words, $\Omega_2$ does not meet a special hyperplane.
\end{proof} 

\begin{remark}
In view of Corollary \ref{cor:sfperiod}, the preceding proposition is in fact a statement
about lattices: it says that there is no special vector $v\in \Lambda_o$ such that 
$u_4(v)=u_5(v)$ and the span $v$ and $u_5(v)$ is a positive definite lattice of rank two
(we recall that $u_k$ is the $(k+1)$th generator of $\mu_3{}^6$ and hence acts on $\Lambda_o$).

\end{remark}

\end{document}